\documentclass[11pt,a4paper]{article}

\usepackage{amsmath,amssymb,amsthm,fullpage}
\usepackage{enumerate}

\usepackage[main=british, german]{babel}
\usepackage[babel=true]{csquotes}

\usepackage{tikz}

\usepackage{graphicx}
\usepackage{bbm} 
\usepackage{subcaption}
\usepackage{hyperref,cleveref}

\title{On the relationship between variable Wiener index and variable Szeged index}
\author{Stijn Cambie\thanks{Department of Mathematics, Radboud University Nijmegen, Postbus 9010, 6500 GL Nijmegen, The Netherlands. This work has been supported by a Vidi Grant of the Netherlands Organization for Scientific Research (NWO), grant number $639.032.614$. Current affiliation: Extremal Combinatorics and Probability Group (ECOPRO), Institute for Basic Science (IBS), Daejeon, South Korea. Email: \href{mailto:stijn.cambie@hotmail.com}{stijn.cambie@hotmail.com}.} \and John Haslegrave\thanks{Mathematics Institute, University of Warwick, Coventry, UK. Supported by the UK Research and Innovation Future Leaders Fellowship MR/S016325/1. Current affiliation: Mathematical Institute, University of Oxford, UK. Email: \href{mailto:j.haslegrave@cantab.net}{j.haslegrave@cantab.net}.}}

\newtheorem*{theorem-non}{Theorem} 

\newtheorem{theorem}{Theorem}[section]

\newtheorem{corollary}[theorem]{Corollary}
\newtheorem{claim}[theorem]{Claim}
\newtheorem{proposition}[theorem]{Proposition}
\newtheorem{conj}[theorem]{Conjecture}
\theoremstyle{definition}
\newtheorem{defi}[theorem]{Definition}

\DeclareMathOperator{\Sz}{Sz}
\DeclareMathOperator{\diam}{diam}
\newcommand{\Sza}{\Sz^\alpha}
\newcommand{\Wa}{W^\alpha}
\newcommand{\ie}{i.e.\ }
\newcommand{\eg}{e.g.\ }

\newcommand*{\myproofname}{Proof}
\newenvironment{claimproof}[1][\myproofname]{\begin{proof}[#1]}{\end{proof}}

\begin{document}
	\maketitle
	
	\begin{abstract}
	We resolve two conjectures of Hri\v{n}\'{a}kov\'{a}, Knor and \v{S}krekovski (2019) concerning the relationship 
	between the variable Wiener index and variable Szeged index for a connected, non-complete graph, one of which would imply the other. 
	The strong conjecture is that for any such graph there is a critical exponent in $(0,1]$, below which the variable Wiener index is larger and
	above which the variable Szeged index is larger. The weak conjecture is that the variable Szeged index is always larger for any exponent 
	exceeding $1$. They proved the weak conjecture for bipartite graphs, and the strong conjecture for trees.
	
	In this note we disprove the strong conjecture, although we show that it is true for almost all graphs, and for bipartite and block graphs.
	We also show that the weak conjecture holds for all graphs by proving a majorization relationship.
	
	
	\end{abstract}

\section{Introduction}\label{sec:intro}
Topological indices of graphs originate from chemical graph theory, and are descriptors of organic molecules that depend only on the graph of atoms and bonds and not on their physical arrangement. The oldest and most important of these is the Wiener index, or total distance, which was introduced by Wiener in 1947 and shown to accurately predict alkane boiling points \cite{Wie47}. Here, as is usual, the index is applied to the unweighted graph of carbon-carbon bonds, neglecting hydrogen atoms. Wiener subsequently established correlations between the Wiener index and other physical quantities (see \eg \cite{Wie48}). The Wiener index was rediscovered by Rouvaray \cite{Rou75} in the 1970s, and applied to more general chemical graphs. It was independently introduced to the mathematics literature for trees by Zelinka \cite{Zel68}, and subsequently extended to general graphs by Entringer, Jackson and Snyder \cite{EJS76}. Since the late 1970s, it has seen an explosion of interest in both fields; see \eg the survey article \cite{DEG01} from the mathematical side, and the collection \cite{50years} from the chemical side. Its importance has led to the study of many other topological indices, such as the Randi\'c index \cite{Ran75}, first and second Zagreb indices \cite{GT72, GRTW75}, and the Hosoya index \cite{Hos71}.

Among topological indices, the Wiener index has a particularly simple and natural definition: it is the sum, over all unordered pairs of vertices, of the graph distance between those vertices (hence the alternative name of ``total distance''; it is linearly related to the average distance when the order is fixed). We define the Wiener index only when the graph is connected, so all distances are finite; this is certainly the case for molecular graphs.

As a special case of the alternative method to compute $W(G)$ by use of a canonical metric representation of a graph, as defined in~\cite{Kla06}, it is well known that for a tree $T$ one can compute the total distance as a sum of a certain quantity over the edges instead of the sum of distances over all pairs of vertices. For each edge $e=uv$, write $n_u(v)$ for the number of vertices strictly closer to $u$ than to $v$ (including $u$ itself), and $n_v(u)$ for the number strictly closer to $v$ than to $u$. Then
\begin{equation}W(T)=\sum_{uv\in E(T)}n_x(y)\cdot n_y(x).\label{wiener-tree}\end{equation}
This fact, which follows from a simple double-counting, was observed by Wiener in his original paper \cite{Wie47}, where only trees were considered.
However, \eqref{wiener-tree} fails to hold for most other graphs, owing to the fact that shortest paths are typically not unique. 

For a general graph $G$, the right-hand side of \eqref{wiener-tree} denotes a different topological index, introduced by Gutman \cite{Gut94} and subsequently called the Szeged index: 
\[\Sz(G):= \sum_{uv \in E(G)} n_u(v)\cdot n_v(u).\] 
In general it is not necessarily true that $\Sz(G)=W(G)$, but the inequality $\Sz(G) \ge W(G)$ holds for all connected graphs \cite{KRG96}. The graphs where equality occurs were classified by Dobrynin and Gutman \cite{DG95}; these are precisely the block graphs (also known as clique trees), that is, graphs where every maximal biconnected subgraph is complete. The difference between the Szeged and Wiener indices can take any non-negative integer value other than $1$ or $3$ \cite{NKA11}.

The study of topological indices encompasses variable forms, which depend not only on the graph $G$ but on a parameter often denoted $\alpha$; the terms which make up the index in question are raised to the power $\alpha$ before summing. This allows the sensitivity of the topological index to extreme values to be altered. For example, the variable form of the Randi\'c index was introduced in \cite{BE98}, and the variable Zagreb indices in \cite{MN04}.

Hri\v{n}\'{a}kov\'{a}, Knor and \v{S}krekovski~\cite{HKS19} generalized the Wiener and Szeged indices to analogous variable forms, and studied the relation between the variable Wiener index and variable Szeged index.
For a parameter $\alpha$, these two graph parameters are defined as respectively 
\[W^{\alpha}(G):= \sum_{u,v \in V} d(u,v)^{\alpha}\quad\text{and}\quad\Sz^{\alpha}(G):= \sum_{uv \in E(G)} (n_u(v)\cdot n_v(u))^{\alpha}.\]
They gave the following two conjectures (in these conjectures, and the rest of this paper, graphs are tacitly assumed to be connected).
\begin{conj}[{\cite[Conjecture 5]{HKS19}}]\label{strong-conj}For every non-complete graph $G$ there is a constant $\alpha_G\in(0,1]$ such that
\begin{align*}
\Sza(G)&>\Wa(G)\quad&\text{if }\alpha>\alpha_G;\\
\Sza(G)&=\Wa(G)\quad&\text{if }\alpha=\alpha_G;\\
\Sza(G)&<\Wa(G)\quad&\text{if }\alpha<\alpha_G.
\end{align*}
\end{conj}
\begin{conj}[{\cite[Conjecture 6]{HKS19}}]\label{weak-conj}For every non-complete graph $G$ and every $\alpha>1$ we have $\Sza(G)>\Wa(G)$.
\end{conj}
Note that Conjecture \ref{strong-conj}, since it requires $\alpha_G\leq 1$, implies Conjecture \ref{weak-conj}. Complete graphs are excluded from the conjectures since if $G$ is complete we have $\Sza(G)=\binom{|G|}{2}=\Wa(G)$ for every $\alpha$.

If we consider $(d(u,v))_{u,v \in V(G)}$ and $(n_u(v)\cdot n_v(u))_{uv \in E(G)}$ as two sequences (of positive integers) $(x_i)_i$ and $(y_i)_i$ derived from a graph, then we are comparing the sums $\sum f(x_i)$ and $\sum f(y_i)$ where $f(x)=x^{\alpha}$. A famous inequality related to such sums is Karamata's inequality~\cite{Karamata}, which is a generalization of the better known inequality of Jensen.
Hri\v{n}\'{a}kov\'{a} et al.~\cite{HKS19} used Karamata's inequality to derive inequalities between $\Sz^{\alpha}(G)$ and $W^{\alpha}(G)$ for trees and bipartite graphs (when $\alpha>1$), proving that the strong conjecture holds for trees and the weak conjecture for bipartite graphs.

In Section~\ref{sec:weak6}, we prove that the weak conjecture is indeed true by a small adaptation of Karamata's inequality combined with a majorization result. In Section \ref{sec:counter} we describe a family of counterexamples to the strong conjecture. However, we are able to prove the strong conjecture for some special families of graphs in Section \ref{sec:strong}. In particular we show that it is true for bipartite graphs, block graphs, and all graphs in a moderately sparse regime. We show that it is true for almost all graphs in a very strong sense: we prove that it is true with high probability for the random graph $G(n,m)$ (that is, a uniformly random choice from the set of all connected graphs with $n$ vertices and $m=m(n)$ edges), even if $m$ is chosen adversarially.

Since this work was submitted, a different and independent proof of the weak conjecture by Vuki\'{c}evi\'{c} and Bulatovi\'{c} has appeared \cite{VB22}.

\section{Proof of the weak conjecture}\label{sec:weak6}
\subsection{An adaptation of Karamata's inequality}\label{sec:Karamata_revisited}
Karamata's inequality generalises Jensen's inequality for convex functions to majorizing sequences. In this article, we will find it convenient to define the notion of majorizing sequences in a slightly more general way, in the sense that the sequences do not need to have the same length or sum. Throughout the whole article, a sequence will always be ordered (from largest to smallest).
\begin{defi}[majorizing sequences]
	Let $(x_i)_{1\le i \le n}$ and $(y_i)_{1\le i \le m}$ be two sequences.
	Then the sequence $(x_i)_i$ \textit{majorizes} $(y_i)_i$ if for every positive integer $k$ it is true that
	$\sum_{1\le i \le k} x_i \ge \sum_{1\le i \le k} y_i$.
	In this summation $x_i$ and $y_j$ are taken to be zero if $i>n$ or $j>m$.
\end{defi}
With this definition in place, Karamata's inequality is as follows.
\begin{theorem}[Karamata's inequality]\label{karamata}Suppose $(x_i)_{1\le i\le n}$ and $(y_i)_{1\le i\le n}$ are decreasing sequences of equal length, taking values in some interval $I$, with $(x_i)_i$ majorizing $(y_i)_{i}$. Let $f:I\to\mathbb R$ be a convex function. If either of the following conditions are satisfied:
\begin{enumerate}[(i)]
\item\label{kar-orig} $\sum_{i=1}^n x_i=\sum_{i=1}^n y_i$ or
\item\label{kar-inc} $f$ is increasing on $I$,
\end{enumerate}
then $\sum_{i=1}^n f(x_i)\geq\sum_{i=1}^n f(y_i)$. Furthermore, provided $f$ is strictly convex, we have equality only if the two sequences are identical.
\end{theorem}
The inequality under condition \eqref{kar-orig} is Karamata's original formulation \cite{Karamata}; the sufficiency of the alternative condition \eqref{kar-inc} is well known and may be proved in the same way (see, for example, \cite[Theorem 12.7]{PPTbook}).

As a corollary, we obtain the following modified version for integer sequences.
\begin{corollary}\label{thr:Karamata_adapted_integers}
	Let $\vec{x}=(x_i)_{1\le i \le n}$ and $\vec{y}=(y_i)_{1\le i \le m}$ be two sequences of nonnegative integers such that the former majorizes the latter and let $t= \sum_i x_i - \sum_i y_i$.
	Let $f$ be a function satisfying $f(0)=0, f(1)=1$ and $f'(x),f"(x) \ge 0$ for all $x \in \mathbb R^+$.
	Then 
	\[\sum_i f(x_i) \ge \sum_i f(y_i) + t.\]
	If the two sequences are not equal (which is certainly the case when $t>0$) and $f$ is strictly convex, the inequality is strict as well.	
\end{corollary}
\begin{proof}
	We can extend the sequence $y_i$ with $t$ values equal to $1$ in the right spot (\ie such that the sequence is still ordered) and add zeros to one of the two sequences to produce sequences $\vec{x}',\vec{y}'$ of the same length. Suppose that $y_k$ is the last positive term of $\vec{y}$, so that the $t$ extra terms become $y'_{k+1},\ldots,y'_{k+t}$. Note that $\sum_{i=1}^ky'_i\leq\sum_{i=1}^k x'_i$, and for $1\leq s\leq t$, we have $\sum_{i=1}^{k+s}y'_i=s+\sum_{i=1}^{k}y_i$. If $x'_{i+s}>0$ then we have 
	$\sum_{i=1}^{k+s}x'_i\geq s+\sum_{i=1}^{k}x'_i\geq\sum_{i=1}^{k+s}y'_i,$ and if not we have 
	$\sum_{i=1}^{k+s}x'_i=\sum_{i\geq 1}x'_i\geq\sum_{i=1}^{k+s}y'_i.$
	Consequently the updated sequence $\vec{x}'$ does majorize $\vec{y}'$ and so Karamata's inequality \eqref{kar-orig} says that 
	\[\sum_i f(x_i) \ge \sum_i f(y_i) + t.\qedhere\]
\end{proof}

\subsection{A majorization relationship}

	As the reader may expect, the crucial fact we will need is the following majorization between the two sequences derived from the graph. We will write $m=|E(G)|$ and $N=\binom{|G|}{2}$ for the natural lengths of the two sequences; we will frequently find it convenient to extend $(n_e)_{e\in E(G)}$ to a sequence of length $N$ by adding zero terms.
	
	\begin{proposition}\label{prop:n_maj_d}
		For every graph $G$, the sequence $(n_u(v)\cdot n_v(u))_{uv\in E(G)}$ majorizes the sequence $(d(u,v))_{u,v \in V}$.
	\end{proposition}
	\begin{proof}
		For every edge $uv$, let $n_{uv}=n_u(v)\cdot n_v(u)$.
		Let $(n_i)_{1 \le i \le m}$ (where $m=\lvert E(G) \rvert$) be the ordered sequence (from largest to smallest) of $(n_e)_e$ and let $(d_i)_i$ be the ordered sequence of the distances $(d(u,v))_{u,v \in V}$.
		The following claim essentially gives the intuition behind the majorization of the sequences, the property that the $n_i$ are larger than $d_i$ in some structured way.
		
		\begin{claim}\label{clm:d<n*n}
			For every pair of vertices $x$ and $y$ and every edge $e=uv$ on a shortest path between $x$ and $y$, we have $d(x,y) \le n_e.$
		\end{claim}
		
		\begin{claimproof}
			Note that there are $d(x,y)+1$ vertices on a shortest path from $x$ and $y$, all of them belonging to $N_u(v)$ or $N_v(u)$. Hence $d(x,y)+1 \le n_u(e)+n_v(e)$, from which $d(x,y) \le n_u(v)\cdot n_v(u)=n_e$ follows.
		\end{claimproof}
		
		This implies that $d(x,y)=\sum_{e \in P} \lambda_e n_e$, where $P$ is a shortest path (containing $d(x,y)$ edges by definition) between $x$ and $y$ and $\lambda_e = \frac{1}{n_e} \le \frac 1 {d(x,y)}$ (by Claim~\ref{clm:d<n*n}), so $\sum_{e \in P} \lambda_e \le 1.$
		
		As a consequence, we have 
		\begin{align*}
		\sum_{i=1}^k d_i &= \sum_{i=1}^k \sum_{e \in P_i} \lambda_e n_e\\
		&= \sum_{e \in E} \Lambda_e n_e,
		\end{align*}
		where $P_i$ is the chosen shortest path corresponding to the distance $d_i$.
		Here every $\Lambda_e$ equals the number of shortest paths $P_i$ through $e$ (so at most $n_e$) times $\lambda_e$, which is bounded by $1$. Also we know $\sum_{e \in E} \Lambda_e = \sum_{i=1}^k \sum_{e \in P_i} \lambda_e \le k$.
		This implies $\sum_{e \in E} \Lambda_e n_e \le \sum_{i=1}^k n_i$.
	\end{proof}
	
	As a consequence, we note that Conjecture \ref{strong-conj} is true in a slightly stronger form.
	\begin{theorem}
		For every $\alpha>1$ and graph $G$ which is not a complete graph, one has $\Sza(G)-\Wa(G)>\Sz(G)-W(G)\geq 0$.
	\end{theorem}
	\begin{proof}
		Since $\lvert E(G) \rvert \neq\binom{n}{2}$, the sequences $(n_i)_i$ and $(d_i)_i$ are not equal. Let $t=\sum_i n_i - \sum_i d_i\geq 0$. 
		Since $f(x)=x^{\alpha}$ satisfies $f(0)=0$ and $f(1)=1$ and is a positive, strictly convex increasing function for $x\ge 0$ and $\alpha>1$, by applying Corollary~\ref{thr:Karamata_adapted_integers} we have $\Sz^{\alpha}(G)>W^{\alpha}(G)+t$.
	\end{proof}

\section{Graphs which satisfy the strong conjecture}\label{sec:strong}

	In Section~\ref{sec:weak6}, we noted that Conjecture \ref{weak-conj} was essentially an inequality for majorizing sequences instead of an inequality on graphs, once one observes Proposition~\ref{prop:n_maj_d} is true. 
	This cannot be the case for Conjecture \ref{strong-conj}.
	One can take the sequences $\vec{x}=\{625,81,81,16\}$ and compare with $\vec{y}=\{256,256,256,1\}$, the first one majorizing the second, but $\sum x_i ^{\alpha} = \sum y_i ^{\alpha}$ being true for $0, \frac 14$ and a number close to $0.88$.
	
	Fix a connected, non-complete graph $G$ with $n$ vertices and let $h(\alpha)=\Sza(G)-\Wa(G)$.
	Since $h(\alpha)$ is a continuous function with $h(0)<0$ and $h(1) \ge 0$, by the intermediate value theorem there is at least one value of $\alpha$ for which $h(\alpha)=0$ (and at least one such value lies in $(0,1]$); the strong conjecture is therefore equivalent to $\alpha$ being unique. We therefore give a sufficient condition, in terms of majorization, for this to be the case.
	\begin{theorem}\label{thr:majorization_x^a_case}
		Let $\alpha$ be a value such that $\Sza(G)=\Wa(G)$ for a non-complete graph $G$.
		If the sequence $(n_i^{\alpha})_i$ does majorize the sequence $(d_i^{\alpha})_i$, then $\alpha$ is unique.
	\end{theorem}
	\begin{proof}
	Assume $\beta$ also satisfies $\Sz^{\beta}(G)=W^{\beta}(G)$. 
	Let $c= \frac{\beta}{\alpha}$; by \cite[Proposition 8]{HKS19} we have $\alpha,\beta>0$ and so $c>0$.
	Then 
	\begin{equation}\label{eq:x^c}
	\sum_i (n_i^{\alpha})^c = \sum_i (d_i^{\alpha})^c,
	\end{equation}
	where we extend the former sum with zero terms to give sequences of equal length.
	For $c>1$ the function $f(x)=x^c$ is strictly convex, while for $c<1$ it is strictly concave.
	By Theorem \ref{karamata} \eqref{kar-orig}, in both cases \eqref{eq:x^c} would not be true, so this implies $c=1$ and hence $\alpha=\beta.$
	\end{proof}	
	Together with Proposition~\ref{prop:n_maj_d}, this immediately gives the following.
	\begin{corollary}If $G$ is a non-complete block graph then $G$ satisfies Conjecture \ref{strong-conj} with $\alpha_G=1$.\end{corollary}
	
	As further examples for the the applicability, let us note that either of the following conditions on $G$ is sufficient (but not necessary) to conclude that $(n_i^{\alpha})_i$ majorizes $(d_i^{\alpha})_i$, where $\alpha$ satisfies $h(\alpha)=0$.
	\begin{enumerate}[I]
		\item\label{cond1} There exists some index $j$ such that $n_i\geq d_i$ if $i\leq j$ and $n_i\leq d_i$ if $i>j$.
		\item\label{cond2} For every $1 \le j \le m$, it holds that $\prod_{i=1}^{j} n_i \ge \prod_{i=1}^{j}d_i$. 
	\end{enumerate}
	To see that \ref{cond2} is sufficient, note that it implies that $(\log n_i)_{1\leq i\leq m}$ majorizes $(\log d_i)_{1\leq i\leq m}$, and hence (since $f(x)=\exp(\alpha x)$ is increasing and convex) that $(n_i^\alpha)_{1\leq i\leq m}$ majorizes $(d_i^\alpha)_{1\leq i\leq m}$. Since we also have $\sum_{i=1}^k n_i^\alpha=\Sza(G)=\Wa(G)\geq\sum_{i=1}^k d_i^\alpha$ for each $k>m$, we have the required majorization property.
	
	For graphs satisfying \ref{cond1}, much more is true.
	\begin{theorem}\label{derivative}If $G$ satisfies condition \ref{cond1} above, then for every value of $\alpha>0$ satisfying $h(\alpha)\geq h(0)$ we have $h'(\alpha)>0$.
	\end{theorem}
	This implies the strong conjecture holds for $G$, since there is some smallest positive value $\alpha_G\leq 1$ for which $h(\alpha_G)=0>h(0)$, and thereafter the function is strictly increasing so strictly positive. Note that Theorem \ref{derivative} is best possible, in the sense that we may have $h'(0)<0$ (and hence $h'(\epsilon)<0$ with $h(\epsilon)$ arbitrarily close to $h(0)$). For example, this is the case for stars with at least $8$ vertices.
	\begin{proof}Set $\tilde{n}_i$ to be $1$ if $n_i=0$ and $n_i$ otherwise. Since $d_i\geq 1$ for each index $i$, condition \ref{cond1} still applies with $n_i$ replaced by $\tilde{n}_i$, and ensures that $\sum_{i\leq k}(\tilde{n}_i^\alpha-d_i^\alpha)$, considered as a function of $k$ for any fixed $\alpha>0$, is increasing for $k\leq j$ and decreasing thereafter. Consequently, $(\tilde{n}_i^\alpha)$ majorizes $(d_i^\alpha)$ provided 
		\[\sum_{i=1}^{N}\tilde{n}_i^\alpha\geq \sum_{i=1}^{N}d_i^\alpha,\]
		\ie whenever $h(\alpha)+N-e(G)\geq 0$. Since $h(0)=e(G)-N$, this condition becomes $h(\alpha)\geq h(0)$.
		Now provided $\alpha>0$ we have 
		\begin{align*}
		h'(\alpha)&=\sum_{i=1}^{N}(\log n_i\cdot n_i^\alpha-\log d_i\cdot d_i^\alpha)\\
		&=\sum_{i=1}^{N}(\log \tilde{n}_i\cdot \tilde{n}_i^\alpha-\log d_i\cdot d_i^\alpha)\\
		&=\alpha^{-1}\biggl(\sum_{i=1}^{N}f(\tilde{n}_i^{\alpha})-\sum_{i=1}^{N}f(d_i^{\alpha})\biggr),
		\end{align*}
		where $f(x)=x\log x$ for $x\geq 1$. Since this is a strictly convex, increasing function, and the two sequences are distinct (the fact that $\sum_i\tilde{n}_i>\sum_id_i$ implies there is some $i$ with $\tilde{n}_i>d_i$), Theorem \ref{karamata} \eqref{kar-inc} gives $\sum_{i=1}^{N}f(\tilde{n}_i)>\sum_{i=1}^{N}f(d_i)$, as required. 
	\end{proof}
	In particular this implies the strong conjecture holds for several classes of graphs.
	\begin{corollary}\label{bip-diam-2-3}The graph $G$ satisfies Conjecture \ref{strong-conj} in any of the following cases:
	\begin{itemize}
	\item $G$ is bipartite;
	\item $G$ is edge-transitive;
	\item $\diam(G)=2$; or
	\item $\diam(G)=3$ and $m\leq N/2$.
	\end{itemize}
	\end{corollary}
	\begin{proof}If $G$ is bipartite then for each edge $uv$, each vertex is strictly closer to either $u$ or $v$, since equality would create an odd cycle. Thus $n_u(v)\cdot n_v(u)\geq |G|-1\geq d(x,y)$, and \ref{cond1} holds with $j=m$.
	
	If $G$ is edge-transitive then every edge is on a diameter, meaning $n_i\geq \diam(G)\geq d_i$ for $1\leq i\leq m$.
	
	If $\diam(G)=2$ then $d_i\in\{1,2\}$ for each $i$. Take $j$ maximal such that $n_i\geq 2$. Then $n_i\geq d_i$ for $i\leq j$, and $n_i\leq d_i$ for all $i>j$, as required. Finally, if $\diam(G)=3$ and $m\leq N/2$ then we have $d_i\in\{2,3\}$ for each $i\leq m$, and taking $j$ maximal such that $n_i\geq 3$ gives the required property.
	\end{proof}
	
	Another corollary of Theorem~\ref{thr:majorization_x^a_case} is the case when the $n_i$ are ``on average'' at least equal to the diameter, which can be stated more precisely as follows with the use of the notion $\mu_{\alpha}$ for the power mean with exponent $\alpha$, \ie using $N=\binom{n}{2}$ we have \[\mu_{\alpha}(d_1,d_2, \ldots, d_N)= \sqrt[\alpha]{\frac{ \sum_{i=1}^N d_i ^{\alpha}}{N}}.\]
	\begin{corollary}\label{cor:mu_a_ge_diam}
		Let $\alpha$ be a value such that $\Sza(G)=\Wa(G)$ for a non-complete graph $G$. 
		If \[\mu_{\alpha}(d_1,d_2, \ldots, d_N) \ge \left(\frac{m}{N}\right)^{\frac 1{\alpha}} \diam(G),\]
		then this value $\alpha$ is unique.
	\end{corollary}
\begin{proof}
	By Theorem~\ref{thr:majorization_x^a_case} it is sufficient to prove that $(n_i^{\alpha})_i$ majorizes $(d_i^{\alpha})_i$, 
	which is the case if for each $k \le m$ the first $k$ terms of the first sequence are on average at least equal to the largest term of the
	second, namely, $\diam(G)^{\alpha}$.
	As the sequence is ordered, it is sufficient to prove this for $k=m$, \ie
	\begin{align*}
	\sum_{i=1}^m n_i^{\alpha} \ge m \diam(G)^{\alpha} &\Longleftrightarrow \sum_{i=1}^N d_i^{\alpha} \ge m \diam(G)^{\alpha}\\
	&\Longleftrightarrow \sqrt[\alpha]{\frac{ \sum_{i=1}^N d_i ^{\alpha}}{N}}\ge \sqrt[\alpha]{\frac{m}{N}} \diam(G).\qedhere
	\end{align*}
\end{proof}	
Corollary \ref{cor:mu_a_ge_diam} is sufficient to prove Conjecture~\ref{strong-conj} for all graphs lying in some sparse regime.
Intuitively if $N\gg m$, the $n_i$ are on average much larger than the $d_i$ and so the majorization property will hold.

\begin{theorem}\label{thr:sparsecase}
	Conjecture~\ref{strong-conj} is true for graphs with $m$ edges and $n$ vertices whenever $m \le \frac14 (n^{4/3}-n^{1/3})$.
\end{theorem}
\begin{proof}
	By Corollary~\ref{cor:mu_a_ge_diam} we only need to prove that the inequality 
	\[\mu_{\alpha}:=\mu_{\alpha}(d_1,d_2, \ldots, d_N) \le \left(\frac{m}{N}\right)^{\frac 1{\alpha}} \diam(G)\]
	is impossible.
	So assume this is the case, \ie we have $\diam(G) \ge \left(\frac{N}{m}\right)^{\frac 1{\alpha}} \mu_{\alpha} \ge (2 n^{2/3})^{\frac 1{\alpha}} \mu_{\alpha}$.
	Note that for all edges in the middle of a diameter, the $n_i$ are large.
	\begin{claim}
		Let $e_1e_2 \ldots e_{\diam(G)}$ be the path between two vertices $x$ and $y$ for which $d(x,y)=\diam(G)$.
		Then $n_{e_j} \ge \frac{3}{16} \diam(G)^2$ when $\frac{\diam(G)}{4} \le j \le \frac{3\diam(G)}{4}+1.$
	\end{claim}
\begin{claimproof}
	Let $e_j=uv$ with $u=e_j \cap e_{j-1}$ and $v=e_j \cap e_{j+1}$.
	Then $N_u(e_j)$ contains at least the $j$ vertices on the path from $x$ to $u$, and similarly $n_v(e_j)\ge \diam(G)+1-j.$
	Hence $n_u(e_j)n_v(e_j)\ge j(\diam(G)+1-j)\ge \frac{3}{16}\diam(G)^2$.
\end{claimproof}
Taking into account that $\alpha \le 1$ and $\mu_{\alpha} \ge 1$, this implies  
\begin{align*}
\sum_i n_i^{\alpha}	&\ge \frac{\diam(G)}{2} \left(\frac{3}{16}\diam(G)^2\right)^{\alpha}\\
&\ge \frac{3}{32} \diam(G)^{3 \alpha}\\
&\ge  \frac3{4} n^{2} \mu_{\alpha}^{\alpha}\\
&> W^{\alpha}(G).
\end{align*}
This is the desired contradiction, so in the sparse regime Corollary~\ref{cor:mu_a_ge_diam} does apply.
\end{proof}
As a consequence, Conjecture~\ref{strong-conj} is a.a.s.\ true for any random graph of the form $G_{n,p}$ or $G_{n,m}$ (conditional on the graph being connected).
\begin{corollary}
	For any value of $m$ or $p$, Conjecture~\ref{strong-conj} is a.a.s.\ true for any random graph of the form $G_{n,p}$ or $G_{n,m}$ (conditional on the graph being connected).
\end{corollary}
\begin{proof}
	By Theorem~\ref{thr:sparsecase} the result is true in the sparse case, so we only have to consider the dense case where 
	$m> \frac14 (n^{4/3}-n^{1/3})$.
	By a result of Klee and Larman~\cite{KL81}, the diameter is a.a.s.\ at most $3$, and when $m \ge  \frac{N}{3}$ it is a.a.s.\ equal to $2$. Therefore Corollary \ref{bip-diam-2-3} applies in this regime. Finally, the result for $G(n,p)$ follows by conditioning on the number of edges.
\end{proof}

\section{Counterexamples to the strong conjecture}\label{sec:counter}

Let $G_{k,\ell}$ be a graph constructed as follows.
Start with a clique $K_k$ and remove a $k$-cycle $v_1v_2v_3 \ldots v_k$ (a Hamiltonian path) from the clique.
Next connect all its vertices with one endvertex $u_1$ of a path $u_1u_2\ldots u_{\ell+1}$ of length $\ell$ (\ie of order $\ell+1$). See Figure \ref{fig:graph} for a small example.

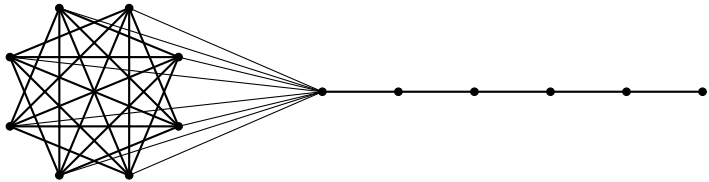
\begin{figure}[ht]
\centering
\begin{tikzpicture}
\foreach \x in {0,45,...,315}{\draw[fill] (\x+22.5:1.2) circle (0.05);
\draw(\x+22.5:1.2) -- (0:3);}
\foreach \x in {0,45,90,135}{\draw[thick] (\x+22.5:1.2) -- (\x+202.5:1.2);}
\draw[thick] (22.5:1.2) -- (112.5:1.2) -- (202.5:1.2) -- (292.5:1.2) -- cycle;
\draw[thick] (67.5:1.2) -- (157.5:1.2) -- (247.5:1.2) -- (337.5:1.2) -- cycle;
\draw[thick] (22.5:1.2) -- (157.5:1.2) -- (292.5:1.2) -- (67.5:1.2) -- (202.5:1.2) -- (337.5:1.2) -- (112.5:1.2) -- (247.5:1.2) -- cycle;
\foreach \x in {3,...,8}{\draw[fill] (0:\x) circle (0.05);}
\draw[thick] (0:3) -- (0:8);
\end{tikzpicture}
\caption{The graph $G_{k,\ell}$ for $k=8$ and $\ell=5$.}\label{fig:graph}
\end{figure}

Now we compute $\Wa(G_{k,\ell})$ and $\Sza(G_{k,\ell})$ (assuming $k\ge 6$).

\subsection{Computation of \texorpdfstring{$\Wa(G_{k,\ell})$}{variable Wiener} and \texorpdfstring{$\Sza(G_{k,\ell})$}{variable Szeged indices}}

The number of edges of the graph equals $\binom{k}{2}+\ell$.
There are $k$ pairs of vertices which are at distance $2$ in the clique.
For each $i\leq \ell+1$, there are $k$ pairs with one vertex from the clique together with $u_i$ which are at distance $i$ from each other.
There are also $\ell+1-i$ pairs of vertices, both belonging to the path, which are at distance $i$ from each other; these are $(u_j,u_{j+i})$ for $1 \le j \le \ell+1-i.$

This implies that 
\[\Wa(G_{k,\ell})=\binom{k}{2}+\ell + (2k+\ell-1)2^\alpha + \sum_{i=3}^{\ell+1} (k+\ell+1-i)i^{\alpha}.\]

Next, we compute $\Sza(G_{k,\ell}).$

For every pair of vertices in the clique of the form $(v_i,v_{i+2})$ (indices modulo $k$), we have 
$N_{v_i}(v_{i+2})=\{v_i,v_{i+3}\}$ and $N_{v_{i+2}}(v_i)=\{v_{i-1},v_{i+2}\}$.
So the corresponding value for $n_{v_iv_{i+2}}=4$;
there are $k$ such edges.

For every edge $v_iv_j$ with $\lvert i-j \rvert \not\equiv 2 \pmod k$ we have $N_{v_i}(v_{j})=\{v_i,v_{j-1},v_{j+1}\}$ and $N_{v_j}(v_i)=\{v_j,v_{i-1},v_{i+1}\}, $ so $n_{v_iv_{j}}=9$.
There are $\binom{k}{2}-2k$ edges of this form.

For every edge $u_1v_i$, we have $N_{v_i}(u_1))=\{v_i\}$ and $N_{u_1}(v_i)$ contains all vertices from the path $P_{\ell}$, as well as $v_{i-1}$ and  $v_{i+1}$.
So we have $k$ edges here for which $n_{v_iu_{1}}=\ell+3.$

For every edge of the form $u_iu_{i+1}$ we have $n_{u_iu_{i+1}}=(k+i)(\ell+1-i).$

Taking everything into account, we get 
\[\Sza(G_{k,\ell})=k4^{\alpha}+\left(\binom{k}{2}-2k\right)9^{\alpha} +k(\ell+3)^{\alpha}+
 \sum_{i=1}^{\ell} ((k+i)(\ell+1-i))^{\alpha}.\]
 
\subsection{Some explicit counterexamples}

There are multiple graphs of the form $G_{k,\ell}$ which are counterexamples. All we need is that $k$ is reasonably large, and then there are some corresponding $\ell$ giving three values of $\alpha$ for which $\Sza(G_{k,\ell})-\Wa(G_{k,\ell})=0$.
One example is $k=520$ and $\ell=82$; part of the function $h(\alpha)$ showing that the conjecture does not hold is plotted in Figure \ref{fig:function}.
\begin{figure}[ht]
\centering
\includegraphics{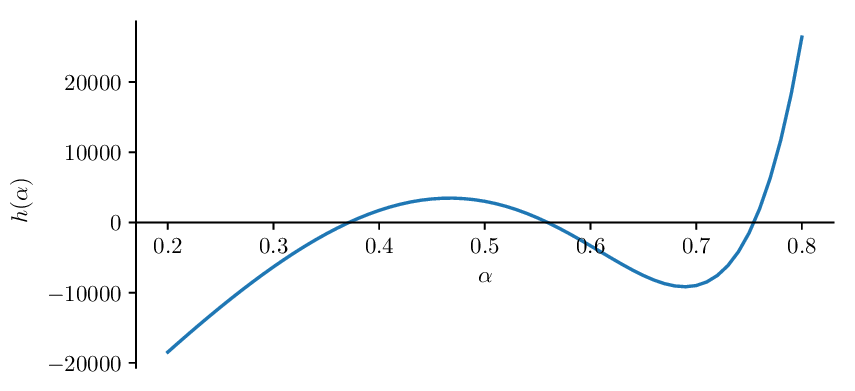}
\caption{The function $h(\alpha)$ plotted for $G_{520,82}$.}\label{fig:function}
\end{figure}

We can actually prove that there are infinitely many counterexamples.
Since our construction is not a block graph, we know that $h(0)<0$ and $h(1)>0$.
For $k$ tending to infinity and $0<\epsilon<0.25$ fixed, $0 \le \alpha \le 1$ and $k^{0.5+\epsilon}< \ell < k^{1-\epsilon}$,
we note that 
\[\Wa(G_{k,\ell})=\binom{k}{2}\left(1+O\left(\frac 1k\right)\right)+k\left(1+O\left(\frac 1{k^\epsilon}\right)\right)\sum_{i=3}^{\ell+1} i^{\alpha}\sim \binom{k}{2} + k \frac{\ell^{\alpha +1}}{\alpha +1}\]
and 
\[\Sza(G_{k,\ell})=\binom{k}{2}\left(1+O\left(\frac 1{k^\epsilon}\right)\right)9^{\alpha}+k^{\alpha}\left(1+O\left(\frac 1{k^\epsilon}\right)\right)\sum_{i=1}^{\ell} i^{\alpha}\sim \binom{k}{2}9^{\alpha} + k^{\alpha} \frac{\ell^{\alpha +1}}{\alpha +1}.\]

This implies that $h(\epsilon)\sim \binom{k}{2}(9^\epsilon-1)>0$ and $h(1-\epsilon)\sim -k \frac{\ell^{2-\epsilon}}{2-\epsilon}<0$.
By applying the intermediate value theorem on $h(0)<0, h(\epsilon)>0, h(1-\epsilon)<0$ and $h(1)>0$ we conclude that there are at least $3$ values $\alpha$ for which $h(\alpha)=0$.

It is not hard to see that the same analysis gives the same conclusion when a matching instead of a cycle (for $k$ even) was removed from the $K_k$, or any $r$-regular graph for fixed $r$ (and $k$ sufficiently large).

\section{Conclusion}

We have resolved both conjectures of \cite{HKS19}. We proved the weaker Conjecture~\ref{weak-conj} by showing majorization of the sequence $(n_i)_i$ with respect to $(d_i)_i$.
The stronger Conjecture~\ref{strong-conj} is not true in general, as there are (infinitely many) sporadic counterexamples such as $G_{520,82}$, as explained in Section~\ref{sec:counter}.

Nevertheless, it is true in many important cases as proven in this note. The conjecture is true for example for the following families:
\begin{itemize}
	\item block graphs;
	\item graphs with diameter $2$;
	\item any $n$-vertex graph with at most $\frac1{2n^{2/3}}\binom{n}2$ edges;
	\item bipartite graphs;
	\item edge-transitive graphs; and
	\item the random graphs $G(n,p)$ and $G(n,m)$ a.a.s.\ (for the entire range of $p$ or $m$).
\end{itemize}

	\bibliographystyle{abbrv}
	\bibliography{Var_W_Sz}
	
\end{document}